\documentclass[12pt]{amsart}
\usepackage{amsmath,amsthm}
\usepackage{geometry}                
\usepackage{graphicx}
\usepackage{amssymb}
\usepackage{epstopdf}
 \usepackage{xcolor}

\usepackage{hyperref}

\usepackage{breakurl}   

\newtheorem{theorem}{Theorem}[section]
\newtheorem{corollary}[theorem]{Corollary}
\newtheorem{lemma}[theorem]{Lemma}
\newtheorem{proposition}[theorem]{Proposition}

\theoremstyle{definition}
\newtheorem{definition}[theorem]{Definition}

\theoremstyle{remark}
\newtheorem{remark}[theorem]{Remark}

\numberwithin{equation}{section}

\begin{document}

\title[] 
 {On  comass  and  stable systolic inequalities}

\author{James J. Hebda }
\address{Department of Mathematics and Statistics, Saint Louis University, St. Louis, MO 63103}
\email{james.hebda@slu.edu} 

 
 \author{Mikhail G. Katz}
\address{ Department of Mathematics, Bar Ilan University, Ramat Gan 52900 Israel}
\email{katzmik@math.biu.ac.il}




\maketitle

\begin{abstract}
We study the maximum ratio  of the Euclidean norm to the comass norm of $p$-covectors in Euclidean $n$-space  and improve the known upper bound found  in the standard references by Whitney and Federer.  We go on to prove stable systolic inequalities when the fundamental cohomology class of the manifold is a cup product of forms of lower degree.  

\end{abstract}

\section{Introduction}

 In the first part of this paper we investigate the maximal ratio $C_{n,p}$  of the Euclidean norm to the comass norm of $p$-covectors in $\mathbf{R}^n$. 
 According to the standard references Federer \cite[1.81]{HF} and Whitney \cite[equation (18) p. 50]{HW},  
\begin{equation*}
 C_{n,p} \leq \binom{n}{p}^\frac 12.
\end{equation*}
In Propositions \ref{p:i1} and \ref{p:i2}, we establish 
 inequalities among the $C_{n,p}$  which lead to an improved upper bound on $C_{n,p}$. 
 As far as we are aware, these two propositions haven't appeared in the earlier literature.
 In Corollary \ref{c:w2}, we obtain an upper bound on the comass of the wedge product of two covectors that improves the upper bound found in  \cite[p. 39]{HF} and \cite[p. 55]{HW}. One of the main results of Goodwillie et al., \cite{GHK}, is an upper bound on the Euclidean  norm of a product of an arbitrary number of $2$-covectors in terms of the product of their comasses. In  Proposition \ref{p:wm}, we generalize this result for wedge products of an arbitrary number of $p$-covectors.
 
 In the last part  of this paper, using the methods of \cite{GHK}, we  apply these results to prove stable systolic inequalities when the fundamental cohomology class of the manifold is a cup product of forms of lower degree.  Background material, including other references, on stable systolic inequalities of this type can be found in \cite{MK}.

\section{Norms on $\Lambda^p(\mathbf{R}^n)^\ast$}

Let $\Lambda^p(\mathbf{R}^n)$  be the vector space  of $p$-vectors in $\mathbf{R}^n$,  and let $\Lambda^p(\mathbf{R}^n)^\ast$  be that of $p$-covectors.  The Euclidean and mass norms of $\xi \in \Lambda^p(\mathbf{R}^n)$ are denoted $|\xi|$ and $\Vert\xi\Vert$ respectively, while the Euclidean and comass norms of $\phi\in \Lambda^p(\mathbf{R}^n)^\ast$ are denoted $|\phi|^\ast$ and $\Vert\phi\Vert^\ast$ respectively. Recall that by definition
\begin{equation*}
\Vert \phi \Vert^\ast = \sup\{\phi(\xi) : \xi\in\Lambda^p(\mathbf{R}^n), \xi \mathrm {\ is\ simple}, |\xi|\leq 1\}.
\end{equation*}
Further  information on mass and comass can be found in \cite[1.81]{HF} or \cite[Chapter I, \S13]{HW}.

\subsection{Norm comparison}

 This section is concerned with the ratio of the  Euclidean norm to the comass norm.

\begin{definition}
\begin{equation*}
 C_{n,p} = \sup \left\{ \frac {|\phi|^\ast}{\Vert \phi \Vert ^\ast} : \phi \in \Lambda^p(\mathbf{R}^n)^\ast, \enspace \phi\neq 0\right\}. 
 \end{equation*}
\end{definition}

\begin{remark}\label{r:1} The following facts about $C_{n,p}$ are well--known \cite{HF,HW}:

\begin{enumerate}
\item Because the Hodge star operation preserves the Euclidean and comass norms, $C_{n,n-p} = C_{n,p}$.
\item Because 1-covectors and $(n-1)$-covectors are simple, $C_{n,1} = C_{n,n-1} =1$ and, because $2$-covectors can be put into a (diagonal) canonical form \cite[Theorem 5.1, p. 24]{SS} ,  $C^2_{n,2}=C^2_{n,n-2} = \lfloor \frac n2\rfloor $.
\item The  upper bound $C_{n,p}^2 \leq \binom{n}{p}$ always holds.
\end{enumerate}

\end{remark}

\begin{proposition}\label{p:i1}
Suppose $1<p<n-1$, then
\begin{equation*}
 C_{n,p}^2  \leq C_{n-1,p-1}^2 + C_{n-1,p}^2.
 \end{equation*}
\end{proposition}
\begin{proof}
Let $\phi \in \Lambda^p(\mathbf{R}^n)^\ast$. Let $e_1,\dots,e_n$ be an orthonormal basis for $\mathbf{R}^n$ with dual basis $e_1^\ast,\dots,e_n^\ast$. We may write
\begin{equation*}
\phi = \psi\wedge e^\ast_n + \chi
\end{equation*}
where $\psi \in \Lambda^{p-1}(\mathbf{R}^{n-1})^\ast$ and $\chi \in \Lambda^p(\mathbf{R}^{n-1})^\ast$. There exist a  simple $(p-1)$--vector $\xi \in \Lambda^{p-1}(\mathbf{R}^{n-1}) \subset \Lambda^{p-1}(\mathbf{R}^n)$ and a simple $p$--vector $\xi^\prime \in \Lambda^{p} (\mathrm{R}^{n-1}) \subset \Lambda^{p} (\mathbf{R}^n)$, both of Euclidean norm 1, such that $\psi(\xi) = \Vert\psi\Vert^\ast$ and
$\chi(\xi^\prime) = \Vert \chi\Vert^\ast$. Then $\xi \wedge e_n$ is a simple $p$--vector  of Euclidean norm 1. Thus
\begin{equation*}
\Vert\phi\Vert^\ast \geq \phi(\xi \wedge e_n) = (\psi\wedge e^\ast_n)(\xi\wedge e_n) + \chi(\xi\wedge e_n)= \psi(\xi) + 0 = \Vert\psi\Vert^\ast 
\end{equation*}
and
\begin{equation*}
\Vert\phi\Vert^\ast \geq \phi(\xi^\prime) = (\psi\wedge e^\ast_n)(\xi^\prime) + \chi(\xi^\prime)=  0 +\chi(\xi^\prime) = \Vert\chi\Vert^\ast. 
\end{equation*}
But
\begin{equation*}
 |\psi|^{\ast} \leq C_{n-1,p-1} \Vert\psi\Vert^{\ast}\quad\mathrm{and}\quad |\chi|^{\ast} \leq C_{n-1,p} \Vert\chi\Vert^{\ast}.
 \end{equation*}
Hence
\begin{eqnarray*}
 |\phi|^{\ast 2} &=& |\psi\wedge e_n^\ast|^{\ast 2} + |\chi|^{\ast 2}=  |\psi|^{\ast 2} + |\chi|^{\ast 2}\\ 
 &\leq&  C_{n-1,p-1}^2 \Vert\psi\Vert^{\ast 2}+ C_{n-1,p}^2 \Vert\chi\Vert^{\ast 2}\\ 
 &\leq&  C_{n-1,p-1}^2 \Vert\phi\Vert^{\ast 2}+ C_{n-1,p}^2 \Vert\phi\Vert^{\ast 2}\\ 
 &=& (C_{n-1,p-1}^2+C_{n-1,p}^2) \Vert \phi\Vert^{\ast 2}.
 \end{eqnarray*}
 Therefore, since this is true for all $\phi$,
 $$ C_{n,p}^2  \leq C_{n-1,p-1}^2 + C_{n-1,p}^2.$$
\end{proof}

The estimate of Remark \ref{r:1}(3) is not optimal. The following
result gives a better upper bound for $C_{n,p}$.

\begin{corollary}\label{c:i1}
\begin{equation*} C_{n,p}^2 \leq \binom{n-2}{p-1} \end{equation*}
\end{corollary}
\begin{proof}
Arrange the constants $C_{n,p}^2$  in a triangular array in the manner of Pascal's Triangle:

\begin{center}
\begin{tabular}{ccccccc}
&&&$C_{2,1}^2$&&&\\
&&$C_{3,1}^2$&&$C_{3,2}^2$&&\\
&$C_{4,1}^2$&&$C_{4,2}^2$&&$C_{4,3}^2$&\\
$C_{5,1}^2$&&$C_{5,2}^2$&&$C_{5,3}^2$&&$C_{5,4}^2$\\
&&&\vdots\\
\end{tabular}
\end{center}
First observe that the first and last terms on each row of the triangle satisfy
\begin{equation*}
 C_{n,1}^2 =1 = \binom{n-2}{0}\quad\mathrm{and}\quad C_{n,n-1} =1= \binom{n-2}{n-2}.
 \end{equation*}
The rest of the proof is by induction,  for by Proposition \ref{p:i1},
\begin{equation*}
 C_{n,p}^2 \leq C_{n-1,p-1}^2 +C_{n-1,p}^2 \leq \binom{n-3}{p-2} +\binom{n-3}{p-1} = \binom{n-2}{p-1} 
 \end{equation*}
by the defining property of the binomial coefficients in Pascal's Triangle, namely, that $\binom{n}{k}+\binom{n}{k+1}=\binom{n+1}{k+1}$.
\end{proof}

Clearly, $\binom{n-2}{p-1} < \binom{n}{p}$ when $ 1 \leq p \leq n-1$. The next proposition gives an even better upper bound on $C_{n,p}$.

\begin{proposition}\label{p:i2}
Let $1\leq k<p<n$, then
\begin{equation*}
C_{n,p}^2 \leq \frac {\binom{n}{p-k}}{\binom{p}{k}} C_{n,k}^2.
\end{equation*}
\end{proposition}

\begin{proof}
Let $\phi = \sum_{I} \beta_I e_I^\ast  \in \Lambda^p(\mathbf{R}^n)^\ast$ where $I$ runs over all multi--indices of length $p$, that is, $I$ is a strictly increasing sequence of $p$ integers between $1$ and $n$.

Fix a multi--index $J$ of length $p-k$ of integers between $1$ and
$n$. Suppose $I$ is a multi--index of length $p$ that contains $J$ as
a subsequence.  In this case write $I \supset J$. Let $I\backslash J$
denote the multi--index of length $k$ complementary to $J$ in
$I$. Thus $e_{I\backslash J} \wedge e_J = \pm e_I \in
\Lambda^p(\mathbf{R}^n)$ where the sign depends on the permutation of
the elements in $I$. Pull out the terms of $\phi$ that contain all the
indices in $J$ by setting
\begin{equation*}
\phi_J  =  \sum_{I\supset J}  \beta_I e_I^\ast.
\end{equation*}
Then let
\begin{equation*}
\psi_J = \sum_{I\supset J} \pm \beta_I e^\ast_{I\backslash J} \in \Lambda^k(\mathbf{R}^n)^\ast,
\end{equation*}
 where  the sign is chosen to agree with the sign in $e_{I\backslash J} \wedge e_J = \pm e_I$, and the sums run over all multi-indices $I$ of length $p$ with $I \supset J$.
 Thus $\phi_J = \psi_J\wedge e_J^\ast$ and $ | \psi_J|^\ast =  \sqrt{\sum_{I \supset J} \beta_I^2}$.
 
 \begin{lemma}
 \begin{equation}\label{e:i2}
   \Vert \phi\Vert^{\ast2} C_{n,k}^2 \geq\sum_{I \supset J} \beta_I^2.
  \end{equation}
 \end{lemma}
 
 \begin{proof}
 Choose a simple $k$--vector  $\xi_J \in \Lambda^k(\mathbf{R}^n)$ such that $|\xi_J|=1$ and $ \psi_J(\xi_J) = \Vert \psi_J\Vert^\ast$. Let $V = \mathrm{span}\{ e_i : i \notin J\} \subset \mathbf{R}^n$. Since by construction $ \psi_J \in \Lambda^k(V)^\ast$, we may assume that $ \xi_J \in \Lambda^k(V) \subset \Lambda^k(\mathbf{R}^n)$.  Since $e_J$ is a simple $p-k$ vector of norm $1$, $ \xi_J \wedge e_J$ is a simple $p$-vector with $| \xi_J \wedge e_J|=1$ . Noting that $e_I^\ast(\xi_J\wedge e_J)=0$ when $I\not\supset J$, it follows that
 \begin{equation*}\Vert \phi\Vert^\ast \geq  \phi( \xi_J \wedge e_J)=  \phi_J( \xi_J \wedge e_J) = (\psi_J\wedge e_J^\ast)( \xi_J \wedge e_J)=\psi_J(\xi_J)e_J^\ast(e_J)  =  \Vert \psi_J\Vert^\ast.
 \end{equation*}
Hence, by definition of $C_{n,k}$,
\begin{equation*}
  \sqrt{\sum_{I \supset J} \beta_I^2} = |\psi_J|^\ast \leq C_{n,k} \Vert \psi_J\Vert^\ast \leq C_{n,k} \Vert \phi \Vert^\ast.
 \end{equation*} 
On squaring this inequality,
 \begin{equation*}
 \quad  \Vert \phi\Vert^{\ast2} C_{n,k}^2 \geq \sum_{I \supset J} \beta_I^2,
 \end{equation*}
 which completes the proof of the lemma.
 \end{proof}
 
 Returning to the proof of Proposition \ref{p:i2}, summing the inequality (\ref{e:i2}) over all $J$ gives
 \begin{equation*}
  \binom{n}{p-k}  C_{n,k}^2 \Vert \phi\Vert^{\ast2} \geq \sum_J\sum_{I \supset J} \beta_I^2 = \binom{p}{k}\sum_I \beta_I^2 =\binom{p}{k} |\phi|^{\ast 2}
  \end{equation*} 
 since in the double sum each $\beta_I^2$ occurs $\binom{p}{k}$ times, because there are $\binom{p}{k}$ multi--indices $J$ contained in each $I$.
 Hence
 \begin{equation*}
  \frac {|\phi|^{\ast 2}}{\Vert \phi\Vert^{\ast2}} \leq   \frac {\binom{n}{p-k}}{\binom{p}{k}} C_{n,k}^2. 
  \end{equation*}
 Therefore, since $\phi$ was arbitrary,
 \begin{equation*}
 C_{n,p}^2 \leq \frac {\binom{n}{p-k}}{\binom{p}{k}} C_{n,k}^2.
 \end{equation*}
 \end{proof}

 \begin{remark}
Proposition \ref{p:i2} improves the upper estimate on $C_{n,p}$ in Remark \ref{r:1}(3). For example, taking $k=1$, 
\begin{equation*}\frac {\binom{n}{p-1}}{\binom{p}{1}} C_{n,1}^2 = \frac 1{n-p+1}\binom{n}{p}.\end{equation*} 
 When $p=2$ and $k=1$,
\begin{equation*}\frac {\binom{n}{1}}{\binom{2}{1}} C_{n,1}^2 = \frac n2 \end{equation*} 
which, according to Remark \ref{r:1}(2), is the optimal bound on $C^2_{n,2}$ when $n$ is even, but not when $n$ is odd.
One may also  check that
\begin{equation*}
\frac 1{n-p+1}\binom{n}{p} \leq \binom{n-2}{p-1}
\end{equation*}
when $1\leq p \leq \frac {n}{2}$ and the inequality is strict unless $p=1$ or $(n,p) = (4,2)$.
Thus Proposition \ref{p:i2} generally provides a better upper bound on $C_{n,p}$ than Corollary \ref{c:i1}.

\end{remark}

\section{Exact values for selected $C_{n,p}$.}

Utilizing  results from the theory of calibrations, we can compute the exact values of some of the $C_{n,p}$.
\begin{proposition}
{$C_{6,3} =2$.}
\end{proposition}
\begin{proof}
By Theorem 4.1 in \cite{FM}, if $\phi \in \Lambda^3(\mathbf{R}^6)^\ast$ has comass $\Vert\phi\Vert^\ast =1$,  then there exists an orthonormal basis $e_1, \dots, e_6$ of $\mathbf{R}^6$  such that
\begin{equation}\label{e:C63}
\quad \phi = e^\ast_1\wedge e^\ast_2 \wedge e^\ast_3+ \mu_1e^\ast_1\wedge e^\ast_5 \wedge e^\ast_6+ \mu_2e^\ast_4\wedge e^\ast_2 \wedge e^\ast_6+ \mu_3e^\ast_4\wedge e^\ast_5 \wedge e^\ast_3
+\mu_4e^\ast_4\wedge e^\ast_5 \wedge e^\ast_6
\end{equation}
with $ 1 \geq \mu_1 \geq \mu_2 \geq |\mu_3|$ and $1 \geq \mu_1^2 +\mu_4^2$. Thus
$$|\phi|^{*2} = 1^2 + \mu_1^2 +\mu_2^2 + \mu_3^2 +\mu_4^2 = 1 + (\mu_1^2 +\mu_4^2) +\mu_2^2 +\mu_3^2 \leq 4.$$
Consequently $|\phi|^* \leq 2 \Vert \phi \Vert^*$ for every $\phi \in \Lambda^3(\mathbf{R}^6)^\ast$.

The value $2$ in this inequality is optimal because by Theorem 2 of Dadok and Harvey\cite{DH}
any $\phi$ of the form (\ref{e:C63}) has comass 1 when 
$$ \mu_1^2 + \mu_2^2 + \mu_3^2 +\mu_4^2 + 2\mu_1\mu_2\mu_3 \leq 1,$$
In particular, on taking $\mu_1=\mu_2 =1$, $\mu_3= -1$, and $\mu_4=0$, Dadok and Harvey's criterion is satisfied. Indeed, by \cite[Theorem 5]{DH}, the special Lagrangian form 
$$\phi = e^\ast_1\wedge e^\ast_2 \wedge e^\ast_3+ e^\ast_1\wedge e^\ast_5 \wedge e^\ast_6+ e^\ast_4\wedge e^\ast_2 \wedge e^\ast_6- e^\ast_4\wedge e^\ast_5 \wedge e^\ast_3$$
has comass 1 and Euclidean norm $|\phi|^* =2$. Therefore $C_{6,3} =2$.
\end{proof}

\begin{proposition}{$C_{8,4} = \sqrt{14}$.}\end{proposition}
\begin{proof}
The estimate for $C_{8,4}$ taking $k=1$ in Proposition \ref{p:i2} is $C_{8,4}^2 \leq 14$. On the other hand, the Cayley 4-form, $\Phi \in \Lambda^4(\mathbf{R}^8)^\ast$,  has comass 1 and Euclidean norm $\sqrt{14}$ \cite{RB, HL}. Thus $C_{8,4} = \sqrt{14}$.
\end{proof}

\begin{proposition}{$C_{7,3}=C_{7,4}=\sqrt 7$.}\end{proposition}
\begin{proof}
Write out the first 8 rows of the triangle of the $C_{n,p}^2$, using all known exact values, including $C_{6,3}^2 = 4$ and $C_{8,4}^2 =14$, to obtain:

\begin{center}
\begin{tabular}{ccccccccccccc}
&&&&&&1\\
&&&&&1&&1\\
&&&&1&&2&&1\\
&&&1&&2&&2&&1\\
&&1&&3&&4&&3&&1\\
&1&&3&&$C_{7,3}^2$&&$C_{7,4}^2$&&3&&1\\
1&&4&&$C_{8,3}^2$&&14&&$C_{8,5}^2$&&4&&1\\
\end{tabular}
\end{center}
Applying Proposition \ref{p:i1}  we find that $C_{7,3}^2 = C_{7,4}^2 \leq 7$ and $C_{7,3}^2 + C_{7,4}^2 \geq 14$, which implies that $C_{7,3}^2 = C_{7,4}^2=7$.
\end{proof}

\section{Wedge products}

\begin{proposition}\label{p:w2}
If $\phi \in \Lambda^p(\mathbf{R}^n)^\ast$ and $\psi \in \Lambda^{n-p}(\mathbf{R}^n)^\ast$, then
 \begin{equation*}
 |\phi\wedge\psi|^\ast \leq C_{n,p}^2  \Vert \phi \Vert^\ast  \Vert \psi\Vert ^\ast.
 \end{equation*}
\end{proposition}
\begin{proof}
Using the Cauchy-Schwarz inequality, 
\begin{equation*}
|\phi\wedge\psi|^\ast = |\langle \phi, *\psi \rangle| \leq |\phi|^\ast |\psi|^\ast \leq C_{n,p} \Vert \phi \Vert^\ast C_{n,n-p} \Vert \psi\Vert ^\ast = C_{n,p}^2  \Vert \phi \Vert^\ast  \Vert \psi\Vert ^\ast
\end{equation*}
since $C_{n,p} =C_{n,n-p}$.
\end{proof}

 Since $C_{p+q,p}^2  < \binom{p+q}{p}$, the next corollary improves the bound 
 \begin{equation*}\Vert\phi\wedge\psi\Vert^\ast \leq  \binom{p+q}{p} \Vert \phi\Vert^\ast\Vert \psi\Vert^\ast \end{equation*} 
 in \cite[p. 39]{HF} or \cite[p. 55]{HW} where $\phi$ and $\psi$ are covectors in degrees $p$ and $q$ respectively.

\begin{corollary}\label{c:w2} 
If $\phi \in \Lambda^p(\mathbf{R}^n)^\ast$ and $\psi \in \Lambda^{q}(\mathbf{R}^n)^\ast$, then
\begin{equation*}
\Vert\phi\wedge\psi\Vert^\ast \leq C_{p+q,p}^2 \Vert \phi\Vert^\ast\Vert \psi\Vert^\ast.
\end{equation*}
\end{corollary}
\begin{proof}
Let $\xi \in \Lambda^{p+q}(\mathbf{R}^n)$ be a simple $(p+q)$-vector with  $|\xi|=1$ such that $(\phi\wedge \psi)(\xi) = \Vert \phi\wedge \psi\Vert^\ast$.
Let $V$ be the $(p+q)$-dimensional subspace of $\mathbf{R}^n$ defined by $\xi$ so that $\xi \in \Lambda^{p+q}(V)$, and
let $\phi_V \in \Lambda^p(V)$ and $\psi_V \in \Lambda^q(V)$ be the restrictions of $\phi$ and $\psi$ to $V$. Then $\Vert \phi_V \Vert^\ast \leq \Vert \phi\Vert^\ast$, $\Vert \psi_V \Vert^\ast \leq \Vert \psi\Vert^\ast$, and
\begin{equation*}
|\phi_V\wedge\psi_V|^\ast = (\phi_V \wedge \psi_V)(\xi) = (\phi\wedge \psi)(\xi) = \Vert \phi\wedge \psi\Vert^\ast.
\end{equation*}
Therefore applying Proposition \ref{p:w2}, 
\begin{equation*}
\Vert \phi\wedge \psi\Vert^\ast= |\phi_V\wedge\psi_V|^\ast \leq C_{p+q,p}^2 \Vert \phi_V\Vert^\ast\Vert\psi_V\Vert^\ast \leq C_{p+q,p}^2 \Vert \phi\Vert^\ast\Vert\psi\Vert^\ast.
\end{equation*}
\end{proof}

\begin{proposition}\label{p:wm}
Let $\phi_1, \dots, \phi_m \in \Lambda^p(\mathbf{R}^{mp})^\ast$, then
\begin{equation*}
|\phi_1\wedge\cdots\wedge\phi_m|^\ast \leq C_{mp,p}^2C_{(m-1)p,p}^2 \cdots C_{2p,p}^2 \Vert \phi_1\Vert^\ast \cdots \Vert \phi_m\Vert^\ast
\end{equation*}
\end{proposition}
\begin{proof}
Applying Corollary \ref{c:w2} successively multiple times, we have that
\begin{eqnarray*}
\Vert\phi_1\wedge\cdots\wedge\phi_m\Vert^\ast  &\leq& C_{mp,p}^2 \Vert \phi_1\wedge\cdots\wedge\phi_{m-1}\Vert^\ast\Vert \phi_m\Vert^\ast\\
 &\leq&C_{mp,p}^2 C_{p(m-1),p}^2 \Vert \phi_1\wedge\cdots\wedge\phi_{m-2}\Vert^\ast\Vert \phi_{m-1}\Vert^\ast\Vert \phi_m\Vert^\ast\\
 &\vdots&\\
 &\leq&C_{mp,p}^2C_{(m-1)p,p}^2 \cdots C_{2p,p}^2 \Vert \phi_1\Vert^\ast \cdots \Vert \phi_m\Vert^\ast
\end{eqnarray*}
Finally, $\Vert\phi_1\wedge\cdots\wedge\phi_m\Vert^\ast  = |\phi_1\wedge\cdots\wedge\phi_m|^\ast $, since $\phi_1\wedge\cdots\wedge\phi_m$ is an $mp$-covector in 
$\mathbf{R}^{mp}$. 
\end{proof}

In the case $p=2$, $C_{2m,2}^2C_{2(m-1),2}^2 \cdots C_{4,2}^2=m!$ by Remark \ref{r:1}(2) (cf. \cite{GHK}.)

\section{Systolic inequalities}

Let $M$ be a smooth  compact orientable manifold of dimension $n$. The image  of the $p$--th integral homology group $H_p(M,\mathbf{Z})$ in $H_p(M,\mathbf{R})$  is the lattice, denoted $\mathsf{L}_p(M)$, in $H_p(M,\mathbf{R})$.  Similarly, the image of the $p$-th integral cohomology group
$H^p(M,\mathbf{Z})$ in $H^p(M,\mathbf{R})$, denoted $\mathsf{L}^p(M)$, is the dual lattice of $\mathsf{L}_p(M)$ 
under the Kronecker pairing \cite[Lemma 15.4.2]{MK}.  In the following, we identify $H^p(M,\mathbf{R})$ with the de Rham cohomology group $H_{dR}^p(M)$
under the canonical isomorphism \cite[Theorem 5.6]{FW} and thereby regard $\mathsf{L}^p(M)$ as a lattice in $H_{dR}^p(M)$ represented by closed forms of integral periods.

If $M$ is endowed with a Riemannian metric $g$, the comass of a differential $p$-form $\phi \in\Omega^p(M)$ is defined to be the supremum of the pointwise comasses, 
\begin{equation*}
\Vert \phi \Vert^\ast_\infty = \sup \{ \Vert \phi_x\Vert^\ast : x \in M\},
\end{equation*}
and the comass $\Vert \alpha\Vert^\ast$ of a de Rham cohomology class $\alpha \in H^p_{dR}(M)$ is defined to be the infimum the comasses of all closed $p$-forms representing $\alpha$.  The mass of a $p$- dimensional current $T$ on $M$ is defined to be
\begin{equation*}
\vert T\vert  = \sup \{ T(\phi) : \phi \in \Omega^p(M) \enspace\mathrm{such\enspace  that}\enspace \Vert\phi\Vert_\infty^\ast \leq 1\},
\end{equation*}
and the stable mass norm $\Vert h\Vert$ of a homology class $h \in H_p(M,\mathbf{R})$ is defined to be the infimum of of $|T|$ over all cycles  $T$ that represent $h$.
The stable mass norm in homology is dual to the comass norm in cohomology. (See \cite[4.35]{MG}, \cite[4.10]{HF2}.)

If the $p$-th Betti number of $M$ is $b$,  let $\lambda_1(\mathsf{L}_p(M)),\dots, \lambda_b(\mathsf{L}_p(M))$ denote the successive minimums of the lattice $\mathsf{L}_p(M)$ relative to the stable mass norm, and let $\lambda_1(\mathsf{L}^p(M)),\dots, \lambda_b(\mathsf{L}^p(M))$ denote the successive minimums of the lattice $\mathsf{L}^p(M)$ relative to the comass norm \cite{JC}.

Our inequalities will need the constant $\Gamma_b$ defined as follows:. 

\begin{definition} Let $b>0$. Then $\Gamma_b$ 
is the supremum of $\lambda_1(L) \lambda_b(L^\ast) $ over all lattices in all $b$-dimensional Banach spaces where $\lambda_1(L)$ is the first successive minimum of $L$ and  $\lambda_b(L^\ast)$ is the $b$-th successive minimum of the dual lattice $L^\ast$.
\end{definition}

\begin{theorem}\label{t:s2}
Let $M$ be a compact orientable manifold of dimension $n$ with $p$-th Betti number $b >0$.  Then for all Riemannian metrics $g$ on $M$,
\begin{equation*}\frac {\mathrm{stsys}_p(M,g)\mathrm{stsys}_{n-p}(M,g)}{\mathrm{vol}(M,g)} \leq C_{n,p}^2 (\Gamma_b)^2.
\end{equation*}
\end{theorem}
\begin{proof}
By Poincar\'e duality, there exist cohomology classes $\alpha \in H^p_{dR}(M)$ and $\beta \in H^{n-p}_{dR}(M)$ such that $\alpha \cup \beta \neq 0$. Since 
the vector space $H^p_{dR}(M)$ is spanned by a set 
consisting of cohomology classes in $\mathsf{L}^p(M)$ of comass at most $\lambda_b(\mathsf{L}^p(M))$,  and $H^{n-p}_{dR}(M)$ is spanned by a set consisting of cohomology classes in $\mathsf{L}^{n-p}(M)$ of comass at most $\lambda_b(\mathsf{L}^{n-p}(M))$,
we may assume that  $\alpha \in \mathsf{L}^p(M)$ with $\Vert \alpha \Vert^\ast \leq \lambda_b(\mathsf{L}^p(M))$ and that $\beta \in \mathsf{L}^{n-p}(M)$ with  $\Vert \beta \Vert^\ast \leq \lambda_b(\mathsf{L}^{n-p}(M))$.

Let $\phi$ and $\psi$ be closed differential forms representing $\alpha$ and $\beta$.
By Proposition \ref{p:w2} applied pointwise, $|\phi\wedge \psi|^\ast \leq C_{n,p}^2\Vert \phi \Vert_\infty^\ast \Vert \psi\Vert_\infty^\ast$. Then
\begin{eqnarray}\label{e:s2}
1 &\leq& \left| \int_M \alpha \cup \beta \right| \leq C_{n,p}^2\Vert \alpha \Vert^\ast \Vert \beta\Vert^\ast \mathrm{vol}(M,g) \\
&\leq& C_{n,p}^2 \lambda_b(\mathsf{L}^p(M)) \lambda_b(\mathsf{L}^{n-p}(M)) \mathrm{vol}(M,g).
\end{eqnarray}
Divide inequality (\ref{e:s2}) by $\mathrm{vol}(M,g)$ and multiply  by the product of
 $\mathrm{stsys}_p (M,g) = \lambda_1(\mathsf{L}_p(M))$ and $\mathrm{stsys}_p (M,g) = \lambda_1(\mathsf{L}_{n-p}(M))$, to obtain
 \begin{eqnarray*}
\frac{\mathrm{stsys}_p (M,g)\mathrm{stsys}_p (M,g)}{\mathrm{vol}(M,g)}&\leq & C_{n,p}^2 \lambda_1(\mathsf{L}_p(M))\lambda_b(\mathsf{L}^{p}(M))
 \lambda_1(\mathsf{L}_{n-p}(M))\lambda_b(\mathsf{L}^{n-p}(M))\\
 &\leq& C_{n,p}^2 (\Gamma_b)^2
\end{eqnarray*}
by the definition of $\Gamma_b$.
\end{proof}

\begin{remark} 
The constant $\Gamma_b$ is a modification of the Hermite constant $\gamma_b$. The modification is needed to take into account non-Euclidean normed spaces that arise naturally on cohomology groups.

The constants $\Gamma_b$ are not known to be a linear function of the Betti number, since Banaszczyk \cite[Corollary 2]{WB} only proves  $\Gamma_b \leq C( b (1+\log b))$ for some constant $C$.  But in a pair of complementary dimensions, as in Theorem \ref{t:s2}, by using the $L^2$ norm on harmonic forms, one can reduce everything to lattices in Euclidean space, thereby replacing  $\Gamma_b$ by the Hermite constant $\gamma_b$  in the general case \cite[Proposition 6]{JH} or  by the  Berg\'e-Martinet constant  $\gamma^\prime_b$ in the $(1,n-1)$ case \cite{BK}. The Hermite and Berg\'e-Martinet constants satisfy $\gamma_b^\prime \leq \gamma_b \leq \frac 32 b$ for $b\geq 2$ \cite[(5.3.4)]{MK}.
\end{remark}

The next result extends the main result in \cite{GHK} for $p=2$ to arbitrary $p$. The reference \cite{MB} gives necessary and sufficient conditions for inequalities of this type.

\begin{theorem}\label{t:smp}
Let $M$ be a compact orientable manifold of dimension $mp$ whose fundamental cohomology class can be expressed as a product of cohomology classes of degree $p$. Then for all Riemannian metrics $g$ on $M$,
\begin{equation*}
 \frac {{\mathrm{stsys}_p(M,g)^m}}{\mathrm{vol}(M,g)} \leq C_{mp,p}^2C_{(m-1)p,p}^2 \cdots C_{2p,p}^2 (\Gamma_b)^m
 \end{equation*}
where $b$ is the $p$-th Betti number of $M$.
\end{theorem}
\begin{proof}
By assumption there exist cohomology classes $ \alpha_1, \dots, \alpha_m \in H^p_{dR}(M)$ such that
 $\alpha_1 \cup\cdots\cup \alpha_m \neq 0$. Again we may assume $\alpha_1, \dots ,\alpha_m \in \mathsf{L}^p(M)$ and that 
$\Vert \alpha_i \Vert^\ast \leq \lambda_b(\mathsf{L}^p(M))$ for all $i$.

Let $\phi_1, \dots, \phi_m$ be closed differential $p$-forms representing $\alpha_1, \dots, \alpha_m$. Applying Proposition \ref{p:wm}
pointwise gives 
\begin{equation*}
|\phi_1 \wedge\cdots\wedge \phi_m|^\ast \leq C_{mp,p}^2C_{(m-1)p,p}^2 \cdots C_{2p,p}^2 \Vert \phi_1\Vert_\infty^\ast \cdots \Vert \phi_m\Vert_\infty^\ast.
\end{equation*}
Then
\begin{equation}
1 \leq \left| \int_M \alpha_1\cup\cdots\cup\alpha_m  \right| \leq C_{mp,p}^2C_{(m-1)p,p}^2 \cdots C_{2p,p}^2 \Vert \alpha_1\Vert^\ast \cdots \Vert \alpha_m\Vert^\ast \mathrm{\ vol}(M,g).
\end{equation}
Multiply this inequality by $\mathrm{stsys}_p(M,g)^m = \lambda_1(\mathsf{L}_p(M,g))^m$ and divide by $\mathrm{vol}(M,g)$ to obtain
\begin{eqnarray*}
\frac{ \mathrm{stsys}_p(M,g)^m}{\mathrm{vol}(M,g) }&\leq& C_{mp,p}^2C_{(m-1)p,p}^2 \cdots C_{2p,p}^2 (\lambda_1(\mathsf{L}_p(M)))^m (\lambda_b(\mathsf{L}^p(M))^m\\
&\leq& C_{mp,p}^2C_{(m-1)p,p}^2 \cdots C_{2p,p}^2 (\Gamma_b)^m
\end{eqnarray*}
by the definition of $\Gamma_b$.
\end{proof}

Using the exact values of $C_{n,p}$ from Section 3, and applying Theorems \ref{t:s2}  or \ref{t:smp}, we obtain: 
 \begin{corollary}
The following statements hold:
 \begin{enumerate}
\item For every Riemannian metric $g$ on a compact orientable 6-dimensional manifold $M$ with third Betti number $b>0$
\begin{equation*} \mathrm{stsys}_3(M,g)^2 \leq 4  (\Gamma_b)^2 \mathrm{vol}(M,g)\end{equation*}

\item For every Riemannian metric $g$ on a compact orientable 7-dimensional manifold $M$ with third Betti number $b>0$
\begin{equation*} \mathrm{stsys}_3(M,g) \mathrm{stsys}_4(M,g) \leq 7  (\Gamma_b)^2 \mathrm{vol}(M,g)\end{equation*}

\item For every Riemannian metric $g$ on a compact orientable 8-dimensional manifold $M$ with fourth Betti number $b>0$
\begin{equation*} \mathrm{stsys}_4(M,g)^2 \leq 14  (\Gamma_b)^2 \mathrm{vol}(M,g)\end{equation*}

\item For every Riemannian metric $g$ on a compact orientable 2m-dimensional manifold $M$ with second Betti number $b>0$
\begin{equation*} \mathrm{stsys}_2(M,g) \mathrm{stsys}_{2m-2}(M,g) \leq m  (\Gamma_b)^2 \mathrm{vol}(M,g)\end{equation*}
\end{enumerate}
\end{corollary}

Equality holds in (4) when $M=\mathbf{CP}^m$ with the standard metric, because in this case $b=1$ and thus $\Gamma_1=1$.  Recall Gromov's stable systolic inequality for $\mathbf{CP}^m$ \cite[Theorem 4.36]{MG}, which asserts that for any Riemannian metric $g$ on  $\mathbf{CP}^m$,
$$ \mathrm{stsys}_2(\mathbf{CP}^m,g) ^m \leq m! \mathrm{vol}(\mathbf{CP}^m,g) $$
with equality holding when $g$ is the standard Fubini--Study metric.  Then, if $g$ is the Fubini--Study metric, $\mathrm{stsys}_{2m-2}(\mathbf{CP}^m,g) = \mathrm{stsys}_2(\mathbf{CP}^m,g)^{m-1}/(m-1)!$, and equality holds in (4).

\end{document}